\documentclass[10pt,a4paper]{amsart}
\usepackage{tikz}
\usepackage{tikz-cd}
\usetikzlibrary{decorations.markings,decorations.pathmorphing,cd}
\usetikzlibrary{arrows}
\usetikzlibrary{calc}
\usepackage{pgfplots}
\pgfplotsset{every axis/.append style={
                    axis x line=middle,    
                    axis y line=middle,    
                    axis line style={-,color=blue}, 
                    xlabel={$x$},          
                    ylabel={$y$},          
            }}
\usepackage{amsmath}
\usepackage{amsthm}
\usepackage{amsfonts}
\usepackage{amssymb}
\usepackage{enumerate}
\usepackage{enumitem}
\usepackage{graphicx}
\usepackage{hyperref}
\usepackage{nicefrac}
\usepackage{cancel}

\newcommand\enet[1]{\renewcommand\theenumi{#1} 
\renewcommand\labelenumi{\theenumi}}

\DeclareMathOperator{\orb}{orb}

\DeclareMathOperator{\SU}{SU}

\def\cT{{\mathcal T}}

\def\ZZ{\mathbb{Z}}

\def\CC{\mathbb{C}}

\def\QQ{\mathbb{Q}}
\def\FF{\mathbb{F}}
\def\QQ{\mathbb{Q}}

\def\PP{\mathbb{P}}

\def\cI{i}

\def\congmap#1{\smash{\mathop{\cong}\limits^{#1}}}

\def\rightmap#1{\smash{\mathop{\rightarrow}\limits^{#1}}}

\newcommand{\Conv}{\mathop{\scalebox{1.5}{\raisebox{-0.2ex}{$\ast$}}}}%

\newtheorem{thm}{Theorem}[section]  
\newtheorem{main-thm}{Theorem}  
\newtheorem{prop}{Proposition}[section]%
\newtheorem{proper}{Condition}[section]%
\newtheorem{main-conj}{Conjecture}[section]%
\newtheorem{cor}{Corollary}[section]
\newtheorem{lemma}{Lemma}[section]

\theoremstyle{remark}
\newtheorem{rem}{Remark}[section]
\newtheorem{notation}{Notation}[section]
\theoremstyle{definition}
\newtheorem{dfn}{Definition}[section]
\newtheorem{exam}{Example}[section]

\makeatletter
\let\c@lemma\c@thm
\let\c@prop\c@thm
\let\c@propdef\c@thm
\let\c@proper\c@thm
\let\c@problem\c@thm
\let\c@conj\c@thm
\let\c@cor\c@thm
\let\c@rem\c@thm
\let\c@dfn\c@thm
\let\c@notation\c@thm
\let\c@exam\c@thm
\makeatother

\title[On the topology of fiber-type curves...]{On the topology of fiber-type curves: \\ an affine Zariski pair of nodal curves}

\author[Jos\'e I. Cogolludo-Agust{\'i}n and Eva Elduque]{J.I.~Cogolludo-Agust{\'i}n and Eva Elduque}
\address{Departamento de Matem\'aticas, IUMA\\
Universidad de Zaragoza \\
C.~Pedro Cerbuna 12 \\
50009 Zaragoza, Spain.} 
\email{jicogo@unizar.es} 

\address{Departamento de Matem\'aticas\\ 
Universidad Aut\'onoma de Madrid, ICMAT \\
28049 Madrid, Spain.}
\email{eva.elduque@uam.es}

\makeindex

\begin{document}

\thanks{The authors are partially supported by the Spanish Government PID2020-114750GB-C31. 
The first author is partially supported by the Departamento de Ciencia, Universidad y Sociedad del 
Conocimiento del Gobierno de Arag{\'o}n (Grupo de referencia E22\_20R ``{\'A}lgebra y Geometr{\'i}a''). The second author is partially supported by the Ram\'on y Cajal Grant RYC2021-031526-I funded by MCIN/AEI /10.13039/501100011033 and by the European Union NextGenerationEU/PRTR} 


\begin{abstract}
In this paper we explore conditions for a curve in a smooth projective surface to have a free product 
of cyclic groups as the fundamental group of its complement. It is known that if the surface is $\PP^2$,
then such curves must be of fiber type, i.e. a finite union of fibers of an admissible map onto a complex 
curve. In this setting, we exhibit an infinite family of Zariski pairs of fiber-type curves, that is, 
pairs of plane projective fiber-type curves whose tubular neighborhoods are homeomorphic, but whose 
embeddings in $\PP^2$ are not. This includes a Zariski pair of curves in $\CC^2$ with only nodes as 
singularities (and the same singularities at infinity) whose complements have non-isomorphic fundamental 
groups, one of them being free. 
Our examples show that the position of nodes also affects the topology of the embedding of projective curves.
Twisted Alexander polynomials with respect to finite $\SU(2)$ representations show to be useful for this purpose,
since all their abelian invariants are the same for both fundamental groups.
\end{abstract}

\maketitle

\section{Introduction}

In a series of recent papers by different
authors~(\cite{MR4349417,Catanese-Fibred,Arapura-toward,ji-Eva-orbifold})
there is a growing interest in studying the connection between the geometric properties of
a smooth connected complex quasi-projective surface~$U$ and its fundamental group.
This paper continues~\cite{ji-Eva-orbifold}, which deals with the case where $\pi_1(U)$ is a free
product of cyclic groups.

One of the main results by the authors in this direction is \cite[Theorem 1.3]{ji-Eva-orbifold},
which gives geometric necessary conditions for a curve inside of a smooth projective surface to
have a free product of cyclic groups as the fundamental group of its complement.
The flavor of this result can be summarized as follows:
if $X$ is a smooth connected complex projective surface and $D\subset X$ is a curve such that
$\pi_1(X\setminus D)$ is a free product of cyclic groups, then there exists a smooth projective curve
$S$ and an admissible (see Definition~\ref{dfn:admissible}) map $F:X\setminus\mathcal B\to S$
inducing an isomorphism $\pi_1(X\setminus D)\to \pi^{\orb}_1(S)$, where $\mathcal B$ is a finite set
of points in $D$ and $\pi^{\orb}_1(S)$ is the orbifold fundamental group for some orbifold structure
on $S$ induced by the multiplicity of the fibers of $F_{|X\setminus D}$. Moreover, $D$ is the closure
of a finite union of fibers of $F$ and of special null-homology horizontal components.
In the particular case where $X=\PP^2$, the curve $D$ is the closure of a finite union of fibers of $F$,
which we refer to as a fiber-type curve.
See Theorem~\ref{thm:main-P2} in section~\ref{sec:fiber-type} for the precise statement when $X=\PP^2$.

The question of whether or not, or to what extent, this is a sufficient condition arises.
For instance, if $X=\PP^2$, $U:=X\setminus D$, and $D$ contains only generic fibers of an admissible map
(satisfying Condition~\ref{cond:nice-pencil} in section~\ref{sec:admissible}), then $\pi_1(U)$ is a
free product of cyclic groups (see~\cite[Theorem 1.2]{ji-Eva-orbifold}). However on the other end, if $D$
contains all atypical fibers of an admissible map, then $\pi_1(U)$ tends to not be a free product of cyclic
groups, but rather a semidirect product of the fundamental groups of two smooth quasi-projective curves
\cite[Lemma 2.18, Corollary 2.19]{ji-Eva-orbifold}. Thus, it is natural to ask if it is possible to formulate
sufficient conditions for the fundamental group of a curve complement to be a free product of cyclic groups
in terms of the topology of the fibers contained in $D$ in the case when some of these fibers contain only
mild singularities (for example nodes).

In section~\ref{sec:threshold}, we describe this problem in more precise terms.
Given an admissible map $F:X\setminus\mathcal B\to S$ and $B_F\subset S$ the set of atypical values of $F$,
we define $\cT_F\subset\mathcal P(B_F)$ the family of subsets $T$ of $B_F$ such that $F_*$ induces an isomorphism
of fundamental groups when restricted to $X\setminus F^{-1}(T)$.
The set of all maximal elements of $\cT_F$ with the inclusion order is called the threshold of atypical values.
If $X=\PP^2$, this set $\cT_F$ completely describes the remarkable phenomenon of $\pi_1(X\setminus D)$ being a
free product of cyclic groups for curves $D$ which are the closure of two or more fibers of~$F$
(see Proposition~\ref{prop:threshold-proj}). A good understanding of the threshold set of atypical values can
provide a source of Zariski pairs of fiber-type curves, where the fundamental group of one of them is a free
product of cyclic groups.

The main result of this paper is shown in section~\ref{sec:zar-pairs} where we explore whether or not the
``free product of cyclic groups'' condition for $\pi_1(X\setminus D)$ could be determined by any topological
information provided by the atypical fibers of $F$ which lie in $D$. To answer this question in the negative,
we give in section~\ref{sec:zar-pair} a special example of a Zariski pair of curves, both of which consist of
two homeomorphic fibers in $\CC^2$ of the same polynomial map, plus the line at infinity. A Zariski pair
(see~\cite{Artal-couples}) refers to two algebraic curves $C_1, C_2\subset \PP^2$ with homeomorphic regular
neighborhoods, but non-homeomorphic embeddings, that is, $(\PP^2,C_1)\not\cong(\PP^2,C_2)$.
In particular, both curves have the same \emph{combinatorial type} (see the discussion after Definition 2
in~\cite{Artal-ji-Tokunaga-survey-zariski}) which includes the same degree and the same local types of singularities.
Analogously, if $D_1, D_2\subset \CC^2$ are two affine curves and one identifies $\CC^2\equiv\PP^2\setminus L_\infty$
where $L_\infty$ denotes the line at infinity, then the curves $D_1$ and $D_2$ form an \emph{affine Zariski pair} if
$D_1\cup L_\infty$ and $D_2\cup L_\infty$ form a Zariski pair in~$\PP^2$.

The example can be described as follows: consider the family of polynomials $f_\lambda(x,y)=x^6+y^6+6xy-\lambda$.
One can construct two curves $D_1=\{f_4\cdot f_{-4}=0\}\subset\CC^2$ and $D_2=\{f_4\cdot f_{4i}=0\}\subset\CC^2$.
Note that each $D_i$ is a union of two disjoint irreducible nodal affine sextics with six nodes each.
One obtains the following in~Theorem~\ref{thm:Zariskipair}.

\begin{thm}
The affine nodal curves $D_1$ and $D_2$ form an affine Zariski pair.
Moreover, $\pi_1(\CC^2\setminus D_1)\not\cong\pi_1(\CC^2\setminus D_2)=\FF_2$ the free group of rank 2.
\end{thm}

This exhibits two fiber-type curves as unions of two fibers of the polynomial map $f:\CC^2\to\CC$ defined by
$f(x,y)=x^6+y^6+6xy$. Each fiber is a sextic with six nodes. The main result shows that the particular choice
of the two fibers given in $D_1$ (resp. $D_2$) produces a non-free (resp. a free) group.

As a word of caution between the affine and projective cases, it is known that projective nodal curves have abelian
fundamental groups (cf.~\cite{Fulton-Fundamental,deligne-groupe}) and hence there cannot be a Zariski pair of
projective nodal curves.

One can also consider the closure of these curves in $\PP^2$. This is done in section~\ref{sec:projective-ZP},
were we prove that the union of the two sextics for both of these curves yields two projective curves of degree
$12$ which also form a Zariski pair. In particular, consider $F_\lambda(X,Y,Z)$ the homogenization of $f_\lambda(x,y)$
and the projective curves $C_1=\{F_4\cdot F_{-4}=0\}\subset\PP^2$ and $C_2=\{F_4\cdot F_{4i}=0\}\subset\PP^2$.
We prove the following in Theorem~\ref{thm:projPair}.

\begin{thm}
The projective curves $C_1$ and $C_2$ form a Zariski pair.
Moreover, $\pi_1(\PP^2\setminus C_2)=\ZZ*\ZZ_6\not\cong\pi_1(\PP^2\setminus C_1)$.
\end{thm}

Note that other examples of Zariski pairs where there is a connection between the topology of their
complement and the position of singular points (including nodes) have been exhibited before
(see~\cite{Tokunaga-Geometry,Bannai-Shirane-Nodal}). However, the remarkable aspect of the present
family of examples stands for the fact that the only special position of singularities is given by the
nodes.

In section~\ref{sec:threshold-example} we classify all possible fundamental groups of complements
of fiber-type curves obtained from fibers of the rational map $F:\PP^2\dashrightarrow\PP^1$ introduced above. This completely
characterizes all the (infinitely many) Zariski pairs whose irreducible components are fibers of this morphism,
and such that the fundamental groups of their complements satisfy that one of them is a free product of cyclic
groups (see Theorem~\ref{thm:allPairs}).

\begin{thm}\label{thm:allPairs-intro}
Let $F:U\to\PP^1$ be given by $F=[x^6+y^6+6xyz^4:z^6]$, where $U=\PP^2\setminus\{[x:y:0]\mid x^6+y^6=0\}$.
Let $B_F\subset\PP^1$ be the set of atypical values of $F$. For every finite set $A\subset\PP^1$, let
$C_A=\overline{F^{-1}(A)}\subset\PP^2$.

Then, the following is a complete list of all Zariski pairs of curves $(C_A,C_{A'})$ such that
the fundamental group of the complement of $C_A$ in $\PP^2$ is a free product of cyclic groups:

\vspace*{7pt}
\begin{center}
\begin{tabular}{|c|c|c|}
$\pi_1(\PP^2\setminus C_{A})$ & $A$ & $A'$ \\
\hline
$\FF_{r+1}*\ZZ_6$ & $\{\varepsilon_14,\varepsilon_24 \cI\}\cup B$ & $\{4\alpha,-4\alpha\}\cup B'$\\
\hline
$\FF_{r+2}$ & $\{\varepsilon_14,\varepsilon_24\cI,\infty\}\cup B$ & $\{4\alpha,-4\alpha,\infty\}\cup B'$\\
\hline
$\FF_{r+2}*\ZZ_6$ & $\{0, \varepsilon_14,\varepsilon_24\cI\}\cup B$ & $\{0,4\alpha,-4\alpha\}\cup B'$\\
\hline
$\FF_{r+3}$ & $\{0,\varepsilon_14,\varepsilon_24\cI,\infty\}\cup B$ & $\{0,4\alpha,-4\alpha,\infty\}\cup B'$\\
\hline
\end{tabular}
\end{center}
\vspace*{7pt}
where $B,B'\subset\PP^1\setminus B_F$ are finite, $r:=\# B= \# B'$,
$\varepsilon_1,\varepsilon_2\in\{\pm 1\}$, and $\alpha\in\{1,\cI\}$.
\vspace*{7pt}
\end{thm}

Finally, in section~\ref{sec:tw-alex} we present a discussion on the (twisted) Alexander-type invariants of the
affine Zariski pair. In particular, it is shown that both curve complements have the same Alexander invariants.
Note that the difference between $D_1$ and $D_2$ is only the position of their nodes (Lemma~\ref{lemma:nodes}).
The position of nodes is known not to have an effect on Alexander polynomials or characteristic
varieties~\cite{Libgober-characteristic}. However, twisted Alexander polynomials with respect to unitary
respresentations in $\SU(2)$ for these groups are different. To our knowledge, this is the first example in the
literature of a Zariski pair of curves whose complements have the same characteristic varieties and whose
fundamental groups have different twisted Alexander polynomials.

\section{Preliminaries}

\subsection{Admissible maps}\label{sec:admissible}
The following definition is found in Arapura~\cite{Arapura-geometry}.

\begin{dfn}\label{dfn:admissible}
	A surjective morphism $F:U\to S$ from a smooth 
	quasi-projective surface $U$ to a smooth projective curve $S$ is admissible if it admits a 
	surjective holomorphic enlargement $\hat F:\hat U\to S$ with connected fibers, 
	where $\hat U$ is a smooth compactification of $U$. 
\end{dfn}

\begin{rem}
	Let $F:U\to S$ be a surjective morphism from a smooth quasi-projective surface to a smooth projective curve. Then $F$ is admissible if and only if $F$ has  connected generic fibers. Indeed, if $F:U\to S$ is admissible and $\hat F:\hat U\to S$ is a holomorphic enlargement with connected fibers, the generic fiber of $\hat F$ is smooth by generic smoothness on the target. Hence, the generic fiber of $F$ is a Zariski open subset of a smooth connected compact curve, which is connected. For the other implication, suppose that $\hat F:\hat U\to S$ is a holomorphic enlargement of $F$, where $\hat U$ is a smooth compactification of $U$. Then, Stein factorization implies that $\hat F$ factors as a map with connected fibers $G:X\to S'$ and a surjective finite morphism $h:S'\to S$, where $S'$ can be taken to be a smooth projective curve. Since $F$ has connected generic fibers, and $F$ also factors through $h$, $h$ must be a branched cover of degree $d=1$, so $h$ is a homeomorphism and thus $\hat F$ has connected fibers.
\end{rem}

	Most of the maps that will appear in this paper will be dominant rational morphisms $F:\PP^2\dashrightarrow\PP^1$ such that the restriction to its maximal domain of definition is an admissible map. Any dominant rational morphism $F$ can be expressed by a surjective morphism $F=[f:g]:U\to\PP^1$ when restricted to its maximal domain of definition $U$, where $f,g\in\CC[x,y,z]$ are homogeneous polynomials of the same degree with no non-constant common factors, and  $U=\PP^2\setminus\mathcal B$, where $\mathcal B$ is the finite set $V(f)\cap V(g)$. In particular, most of the maps that will appear in this paper will be of the form described below:

\begin{proper}\label{cond:nice-pencil}
	 $F=[f_q^p:f_p^q]:U\to\PP^1$ is such that
	\begin{itemize}
		\item $p,q\in\ZZ_{\geq 1}$, with $\gcd(p,q)=1$,
		\item $f_p,f_q\in\CC[x,y,z]$ are homogeneous polynomials of degrees $p$ and $q$ respectively, with no non-constant common factors,
		\item neither $f_p$ nor $f_q$ is a $k$-th power of another polynomial in $\CC[x,y,z]$ for any $k\geq 2$,
		\item $U=\PP^2\setminus\mathcal B$, where $\mathcal B$ is the finite set $V(f_p)\cap V(f_q)\in\PP^2$,
		\item The map $F$ has no multiple fibers outside $\{[0:1],[1:0]\}$.
	\end{itemize}
\end{proper}

The last point in Condition~\ref{cond:nice-pencil} is necessary (up to change of coordinates in $\PP^1$) for $F$ to be admissible, as shown in the following remark.

\begin{rem}\label{rem:multipleFibers}
	 Let $F:U\to\PP^1$ be the restriction to its maximal domain of definition of a dominant rational map $F:\PP^2\dashrightarrow \PP^1$. If $F:U\to\PP^1$ is admissible, \cite[Proposition 2.8]{ji-Libgober-mw} implies that the number of multiple fibers of $F$ does not exceed $2$, so, after a change of coordinates in $\PP^1$, we may assume that the multiple fibers of $F$ lie over a subset of $\{[0:1],[1:0]\}$.
\end{rem}

\begin{rem}\label{rem:affine-admissible}
Let $f:\CC^2\to\CC$ be a non-constant polynomial map with connected generic fibers. Then, $f$ extends to an 
admissible map $F:\PP^2\setminus\mathcal B=[\overline f(x,y,z):z^d]\to\PP^1$, where $\overline f$ is the homogenization 
of $f$ with respect to the variable $z$, $d$ is the degree of $f$, and $\mathcal B=V(\overline f(x,y,z))\cap V(z)$.
\end{rem}

It is easy to construct admissible maps satisfying Condition~\ref{cond:nice-pencil}, as exemplified by the following result.

\begin{lemma}\label{lem:irreducibleAdmissible}
Let $F=[f_{kq}^p:f_{kp}^q]:U=\PP^2\setminus(V(f_{kq})\cap V(f_{kp}))\to\PP^1$  
such that $k\in\ZZ_{\geq 1}$, $\gcd(p,q)=1$, and $f_{kq}, f_{kp}\in\CC[x,y,z]$ are irreducible homogeneous polynomials of degree $kq$ and $kp$ respectively that are not constant multiples of one another.
Then, $F$ is an admissible map. Moreover, if $k=1$, then $F$ also 
 satisfies Condition~\ref{cond:nice-pencil}.
\end{lemma}
\begin{proof}
Let $\hat F:X\to\PP^1$ be a resolution of indeterminacies of the pencil $F:\PP^2\dashrightarrow\PP^1$. 
In particular, $X$ is a smooth simply connected projective surface. Using Stein factorization and the 
fact that $X$ is simply connected, $\hat F$ factors as $H\circ\hat G$, where $\hat G:X\to\PP^1$ has 
connected fibers and $H:\PP^1\to\PP^1$ is finite, generically $d:1$. Our goal is to show that $d=1$. 
Let $G:= \hat G_{|U}:U\to\PP^1$, which is an admissible map by construction.

The map $H$ (resp. $G$) is of the form $[h_1(x,y):h_2(x,y)]$ (resp. $[f(x,y,z):g(x,y,z)]$, 
where $h_1,h_2\in\CC[x,y]$ (resp. $f,g\in\CC[x,y,z]$) are homogeneous polynomials of the same degree 
with no non-constant common factors. Hence,
$$
f_{kq}^p=h_1(f,g),\quad f_{kp}^q=h_2(f,g).
$$
Since $f_{kq}$ and $f_{kp}$ are irreducible and every homogeneous polynomial in $\CC[x,y]$ decomposes as a product of 
linear forms, we have that there exist $[a:b], [c:d]\in\PP^1$ and $m,l\in\ZZ_{\geq 1}$ such that
$$
h_1(x,y)=(bx-ay)^m,\quad m\mid p,\quad h_2(x,y)=(dx-cy)^l,\quad l\mid q,
$$
and since $h_1$ and $h_2$ have the same degree, $m=l$. The condition that $\gcd(p,q)=1$ implies that $m=l=1$. 
Hence, $H$ is a change of coordinates in $\PP^1$, which concludes the proof of the fact that $F$ is admissible. 
By Remark~\ref{rem:multipleFibers}, if $k=1$, $F$ also satisfies Condition~\ref{cond:nice-pencil}, since the existence
of a multiple fiber over a point not in $\left\{[0:1],[1:0]\right\}$ would imply that $F$ does not have connected generic fibers.
\end{proof}

\begin{exam}
The result of Lemma~\ref{lem:irreducibleAdmissible} is no longer true if the irreducibility assumption is dropped. For example, if $f_p=x^p-y^p$ and $f_q=x^q+y^q$ and $p$ is odd, then $[f_p^q:f_q^p]:\PP^2\setminus\{[0:0:1]\}\to\PP^1$ is not an admissible map, since the closure of any of its fibers is a union of lines in $\PP^2$ going through the point $[0:0:1]$. 
\end{exam}

\subsection{Orbifold fundamental groups and morphisms}
Let $S$ be a smooth projective curve of genus $g$. We may endow it with an orbifold structure by choosing 
$\varphi:S\to\ZZ_{\geq 0}$ such that $\varphi(P)\neq 1$ only for a finite number of points. Let 
$\Sigma=\Sigma_0\cup\Sigma_+\subset S$ be the points for which $\varphi(P)= 0$ if $P\in \Sigma_0$ and 
$\varphi(Q)=m_{Q}>1$ if $Q\in \Sigma_+$. We will denote this structure by $S_{(n+1,\bar m)}$, where $n+1=\#\Sigma_0$, 
and $\bar{m}$ is a $(\#\Sigma_+)$-tuple whose entries are the corresponding $m_Q$'s. 

As a word of caution for the following definitions, the term orbifold might be misleading: we do not need to develop any theory of orbifolds or $V$-manifolds in this context, but rather use the orbifold structures to highlight the existence 
of multiple fibers of an admissible map. This will become clear throughout the section.

\begin{dfn}[Orbifold fundamental group of a smooth projective curve]
Let	$S_{(n+1,\bar m)}$ be a smooth projective curve endowed with an orbifold structure. The orbifold fundamental group of $S_{(n+1,\bar m)}$ is defined by
$$
\pi_1^{\orb}(S_{(n+1,\bar m)}):=\pi_1(S\setminus \Sigma)/
\langle \mu_P^{\varphi(P)}, P\in \Sigma\rangle,
$$
where $\mu_P$ is a meridian in $S\setminus \Sigma$ around $P\in \Sigma$, and $\langle \mu_P^{\varphi(P)}, P\in \Sigma\rangle$ is the normal subgroup generated by the $\mu_P^{\varphi(P)}$'s.
\end{dfn} 

Note that if $S$ has genus $g$, $\pi_1^{\orb}(S_{(n+1,\bar m)})$ is hence generated by
$
\{a_i,b_i\}_{i=1,\dots,g} \cup \{\mu_P\}_{P\in \Sigma}
$
and presented by the relations
\begin{equation}
	\label{eq:rels}
	\mu_P^{m_P}=1, \quad \textrm{ for } P\in \Sigma_+, \quad \textrm{ and } \quad 
	\prod_{P\in \Sigma}\mu_P=\prod_{i=1,\dots,g}[a_i,b_i].
\end{equation}
In particular,  if $\Sigma_0\neq \emptyset$, \eqref{eq:rels} shows that 
$\pi_1^{\textrm{orb}}(S_{(n+1,\bar m)})$ is a free product of cyclic groups as follows 
$$
\pi_1^{\textrm{orb}}(S_{(n+1,\bar m)})\cong \pi_1(S\setminus\Sigma_0)\Conv\left(
\Conv_{P\in\Sigma_+}\left(\frac{\ZZ}{m_P\ZZ}\right)\right)\cong
\FF_r*\ZZ_{m_1}*\dots*\ZZ_{m_s},
$$
where $r=2g-1+\# \Sigma_0=2g+n$, $s=\# \Sigma_+$, and $\bar m=(m_1,\dots,m_s)$. Note that any finitely generated free product of cyclic groups can be realized as an orbifold fundamental group of a smooth projective curve in this way.

\begin{dfn}[Orbifold morphism]
	Let $U$ be a smooth algebraic variety. A dominant algebraic morphism $F:U\to S_{(n+1,\bar m)}$ defines an 
	orbifold morphism if for all $P\in S$ such that $\varphi(P)>0$, the divisor $F^*(P)$ is a 
	$\varphi(P)$-multiple. 
\end{dfn}

\begin{dfn}[Maximal orbifold structures]
	Let $F:U\to S_{(n+1,\bar m)}$ be an orbifold morphism. The orbifold $S_{(n+1,\bar m)}$ is said to be maximal (with respect to $F$) 
	if $F(U)=S\setminus \Sigma_0$ and for all $P\in F(U)$ the divisor $F^*(P)$ is not an $n$-multiple for any $n > \varphi(P)$. 
\end{dfn}

The following result is well known (see~\cite[Prop. 1.4]{ACM-multiple-fibers}, for example)
\begin{rem}\label{rem:inducedorb}
Let $F:U\to S_{(n+1,\bar m)}$ be an orbifold morphism. Then, $F$ induces a morphism
$
F_*:\pi_1(U)\to\pi_1^{\orb}(S_{(n+1,\bar m)}).
$
Moreover, if the generic fiber of $F$ is connected, then $F_*$ is surjective.
\end{rem}

We end this section with important notational conventions that will be used throughout the rest of the paper:
\begin{itemize}
\item Whenever we refer to a dominant algebraic morphism $F:U\to S$ as an orbifold morphism, we will assume that $S$ is endowed with the orbifold structure that turns $F$ into an orbifold morphism in such a way that the orbifold structure is maximal.
\item The only orbifold structures that will appear are the maximal ones. 
\item In particular, if $F:U\to S$ is surjective and $B\subset S$, the notation $F_*:\pi_1\left(U\setminus F^{-1}(B)\right)\to\pi_1^{\orb}(S\setminus B)$ will refer to the morphism endowed by $F_{|U\setminus F^{-1}(B)}:U\setminus F^{-1}(B)\to S$ on (orbifold) fundamental groups, where $S$ is endowed with the maximal orbifold structure with respect to $F_{|U\setminus F^{-1}(B)}$ (which has $\Sigma_0=B$).
\end{itemize}

\subsection{Fiber-type curves}\label{sec:fiber-type}

\begin{dfn}[Fiber-type curve]\label{dfn:fiber-type}
Let $X$ be a smooth projective surface, and let $C$ be a curve in $X$. We say that a $C$ is a fiber-type curve if, for some complement of a finite number of points in $X$ called $U$, there exists an admissible map $F:U\to S$ to a smooth projective curve $S$ such that $C$ is the closure in $X$ of a finite number of fibers of $F$.
\end{dfn}

\begin{rem}\label{rem:affine-fiber-type}
Let $f:\CC^2\to\CC$ be a non-constant polynomial map with connected generic fibers. We will also refer to the union $C$ of a finite number of fibers of $f$ as an \emph{affine fiber-type curve}. Indeed, both points of view agree if we want to study the topology of their complements: $f$ extends to an admissible map $F:\PP^2\setminus\mathcal B=[\overline f(x,y,z):z^d]\to\PP^1$ as in Remark~\ref{rem:affine-admissible}. The complement in $\PP^2$ of the union of $\overline C$ and the line at infinity (which is the closure in $\PP^2$ of the fiber of $F$ over $[1:0]$) coincides with $\CC^2\setminus C$.
\end{rem}

As mentioned in the introduction, the curves in a smooth projective surface $X$ whose complements have a fundamental group which is a free product of cyclic groups are (essentially) fiber-type curves (with perhaps some extra null-homotopic irreducible components). In fact, if $X=\PP^2$, the curves must actually be fiber-type curves, as exemplified by the following result:

\begin{thm}[Main Theorem for curves in $\PP^2$~{\cite[Cor.~3.14]{ji-Eva-orbifold}}]\label{thm:main-P2}
	Let $D$ be a curve in $\PP^2$.
	Suppose that $\pi_1(\PP^2\setminus D)$ is a free product of cyclic groups. 
	Then, there exists $r\geq 0$ and $m_1\geq m_2\geq 1$ with $\gcd(m_1,m_2)=1$ such that $\pi_1(\PP^2\setminus D)\cong\FF_r*\ZZ_{m_1}*\ZZ_{m_2}$. Moreover, there exists $F=[f_{q}^{p}:f_{p}^{q}]:U=\PP^2\setminus\mathcal B\to \PP^1$ an admissible map as in Condition~\ref{cond:nice-pencil}
	 such that:
	\begin{enumerate}
		\enet{\rm(\roman{enumi})}
		\item \label{thm:main-simply-1}
		$F$ induces an orbifold morphism
		$$
		F_|:\PP^2\setminus D\to \PP^1_{(r+1,\bar m)},
		$$
		where $\PP^1_{(r+1,\bar m)}$ is maximal with respect to $F_|$, and $\bar m$ takes the following value: $(m_1,m_2)$ if $m_2\geq 2$, $(m_1)$ if $m_1\geq 2$ and $m_2=1$, and it is empty if $m_1=m_2=1$.
		\item \label{thm:main-simply-2} 
		$F_*:\pi_1(\PP^2\setminus D)\to \pi_1^{\orb}(\PP^1_{(r+1,\bar m)})$ is an isomorphism.
		\item \label{thm:main-simply-3}
		$D=\overline{F^{-1}(\Sigma_0)}\subset\PP^2$ is a fiber-type curve which is the closure of the union of $r+1$ irreducible fibers of $F$.
	\end{enumerate}
	 More concretely,
	\begin{enumerate}
		\item\label{thm:main-simply-4-1} If $m_1>m_2>1$, then $p=m_1$, $q=m_2$ and the pencil $F=[f_{m_2}^{m_1}:f_{m_1}^{m_2}]$ 
		has exactly two multiple fibers corresponding to $[0:1],[1:0]\notin \Sigma_0$.
		\item\label{thm:main-simply-4-2} 
		If $m_1>m_2=1$, then $p=m_1$, $[1:0]\in \Sigma_0$, and the pencil $F=[f_{q}^{m_1}:f_{m_1}^{q}]$ 
		has at least one multiple fiber corresponding to $[0:1]\notin \Sigma_0$.
		\item\label{thm:main-simply-4-3} 
		If $m_1=m_2=1$, then $F=[f_{q}^{p}:f_{p}^{q}]$ has at most two multiple 
		fibers corresponding to $[0:1],[1:0]\in \Sigma_0$.
	\end{enumerate}
\end{thm}

\begin{rem}\label{rem:free}
	This result implies that, if an affine curve satisfies that the fundamental group of its complement in $\CC^2$ is a free product of cyclic groups, then it is a free group, and moreover, the affine curve must consist of a finite union of irreducible fibers of a 
	polynomial map $f:\CC^2\to\CC$.
\end{rem}

The goal of Section~\ref{sec:zar-pairs} is to find Zariski pairs in this setting, where the irreducible components of both curves will be closures of fibers of the same admissible map.

We introduce the following notation.

\begin{notation}
Let $U$ be a smooth quasi-projective surface and let $S$ be a smooth projective curve. If $F:U\to S$ is an admissible map 
 and  $B\subset S$ is a finite set, we let $U_B$ be
 $$U_B=U\setminus F^{-1}(B).$$
\end{notation}

\begin{rem}
Let $F=[f:g]:U=\PP^2\setminus\mathcal B\to \PP^1$ be an admissible map, where $f,g\in\CC[x,y,z]$ are homogeneous polynomials of degree $d\geq 1$ with no non-constant common factors, and $\mathcal B=V(f)\cap V(g)$. Let $B$ be a non-empty finite subset of $\PP^1$, and let $D=\overline{F^{-1}(B)}$. Then $U_B=\PP^2\setminus D$.
\end{rem}

The following result can be found in \cite[Thm.~1.4, Cor.~4.9]{ji-Eva-orbifold}.

\begin{thm}[Addition-Deletion Lemma]
\label{thm:generic-fiber-orbi}
Let $U$ be a smooth quasi-projective surface and let $F:U\rightarrow S$ be an admissible map to a smooth projective curve $S$. 
Assume $B\subset S$, where $\# B=n\geq 1$, and let $P\in S\setminus \left(B_F\cup B\right)$, where $B_F$ is the (finite) set of atypical values of $F$.

Then the following are equivalent:
\begin{enumerate}
 \item\label{1}
$F_*:\pi_1(U_B)\to\pi_1^{\orb}(S\setminus B)$ is an isomorphism,
 \item\label{2} 
$F_*:\pi_1(U_{B\cup\{P\}})\to\pi_1^{\orb}\left(S\setminus (B\cup\{P\})\right)$ is an isomorphism,
\end{enumerate}
and in that case, 
$$\pi_1(U_{B\cup\{P\}})\cong \ZZ*\pi_1(U_B).$$
Moreover, the implication \eqref{2}$\Rightarrow$\eqref{1} also holds if $P\in B_F\setminus B$.
\end{thm}

\begin{rem}
We will refer to \eqref{1}$\Rightarrow$\eqref{2} as the Addition Lemma (which is true for generic fibers) 
and to \eqref{2}$\Rightarrow$\eqref{1} as the Deletion Lemma (which is true for all fibers).
\end{rem}

\section{Threshold of atypical values}\label{sec:threshold}
Let $U$ be a smooth quasi-projective surface and let $S$ be a smooth projective curve. Let $F:U\to S$ 
be an admissible map and let $B\subset S$ be a finite set of points. The goal of this section is to study 
how the presence of atypical values of $F$ in $B$ relates to the induced morphism 
$F_*:\pi_1(U_B)\to \pi_1^{\orb}(S\setminus B)$ being an isomorphism.

To study this problem, we start by defining the set
$$\cT_{F}:=\left\{T\subset B_F \left|
\begin{array}{c}
F_*:\pi_1(U_{\{P\}\cup T})\to\pi_1^{\orb}\left(S\setminus(\{P\}\cup T)\right)\\
\textrm{ is an isomorphism}\end{array}\right.\right\},$$
where $P\in S\setminus B_F$. The following result summarizes some of the properties of $\cT_{F}$.

\begin{lemma}\label{lem:preThreshold}
Let $U$ be a smooth quasi-projective surface and let $S$ be a smooth projective curve.
Let $F:U\to S$ be an admissible map, let $B_F\subset S$ be its set of atypical values,
and let $P\in S\setminus B_F$. The following hold:
\begin{enumerate}
\item\label{pre1} The set $\cT_{F}$ is independent of the choice of $P\in S\setminus B_F$.
\item\label{pre2} If $T\in \cT_{F}$ and $T'$ is such that $T'\subset T\subset B_F$, then $T'\in \cT_{F}$.
\item\label{pre3} $\cT_{F}\neq\emptyset$ if and only if $F_*:\pi_1(U_{P})\to\pi_1^{\orb}(S\setminus\{P\})$ is an isomorphism.
\end{enumerate}
\end{lemma}

\begin{proof}
Let $P,Q\in S\setminus B_F$, let $T\subset B_F$ and suppose that $T\in\cT_F$, that is,
$$
F_*:\pi_1(U_{\{P\}\cup T})\to
\pi_1^{\orb}\left(S\setminus (\{P\}\cup T)\right)\quad\text{is an isomorphism.}
$$
By the Addition Lemma (Theorem~\ref{thm:generic-fiber-orbi}), one can add 
the generic fiber of $F$ at $Q$, that is,
$F_*:\pi_1(U_{\{P,Q\}\cup T})\to \pi_1^{\orb}(S\setminus \{P,Q\}\cup T)$
is an isomorphism. Using now the Deletion Lemma, the fiber at $P$ can be removed 
and $F_*:\pi_1(U_{\{Q\}\cup T})\to \pi_1^{\orb}(S\setminus \{Q\}\cup T)$
is an isomorphism. This proves part~\eqref{pre1}. Part \eqref{pre2} follows directly from the 
Deletion Lemma. Part~\eqref{pre2} implies that $\cT_{F}\neq\emptyset$ if and only if $\emptyset\in\cT_{F}$,
which implies part~\eqref{pre3}.
\end{proof}

The previous result suggests the following definition.

\begin{dfn}
Let $U$ be a smooth quasi-projective surface, let $S$ be a smooth projective curve.
Let $F:U\to S$ be an admissible map and let $B_F\subset S$ be its set of atypical values.
We say that a set $T\subset B_F$ is a
\emph{threshold set of (atypical) values for $F$}
if it is a maximal set of $\cT_{F}$ (with respect to the inclusion of sets).
\end{dfn}

\begin{rem}[Non-empty $\cT_F$, projective case]\label{rem:non-empty-proj}
Let  $F=[f_q^p:f_p^q]:U=\PP^2\setminus\mathcal B\to \PP^1$ be an admissible map satisfying 
Condition~\ref{cond:nice-pencil}. If $P\in \PP^1\setminus B_F$, then 
$F_*: \pi_1(U_{P})\to\pi_1^{\orb}\left(S\setminus\{P\}\right)$ is an isomorphism by 
\cite[Theorem 1.2]{ji-Eva-orbifold}, hence, 
Lemma~\ref{lem:preThreshold}\eqref{pre3} implies~$\cT_F\neq\emptyset$.
\end{rem}

\begin{exam}[Empty $\cT_F$]\label{ex:empty}
Consider $\mathcal B=V(h_2)\cap V(h_4)$, where $h_d$ is an irreducible homogeneous polynomial in 
$\CC[x,y,z]$ of degree $d$ for $d=2,4$. Define $U=\PP^2\setminus\mathcal B$, $S=\PP^1$, and $F:U\to S$
is described as $F=[h_2^2:h_4]$. By Lemma~\ref{lem:irreducibleAdmissible}, $F$ is an admissible map, however
$F_*:G_1:=\pi_1(U_P)\to\pi_1^{\orb}(S\setminus\{P\})=:G_2$ is not an 
isomorphism for any $P\notin B_F$, since the abelianizations $G_1/G'_1=\ZZ_4$, $G_2/G'_2=\ZZ_2$ 
are not isomorphic. Hence, Lemma~\ref{lem:preThreshold}\eqref{pre3} implies~$\cT_F=\emptyset$.
\end{exam}

The purpose of $\cT_F$ is to characterize the property that 
$F_*:\pi_1(U_B)\to \pi_1^{\orb}(S\setminus B)$ defines an isomorphism, regardless of whether or not $B$ 
contains a generic fiber, as the following result shows.

\begin{prop}
\label{prop:BinBF}
Let $U$ be a smooth quasi-projective surface, $S$ a smooth projective curve and $F:U\to S$ be an admissible map. 
Consider $B\subset S$ a finite non-empty set. Then the following are equivalent:
\begin{enumerate}
\item\label{threshold2a}
$F_*:\pi_1(U_B)\to \pi_1^{\orb}(S\setminus B)$ is an isomorphism.
\item\label{threshold3a}
There exists $T\in \cT_{F}$ such that $B\cap B_F\subset T$.
\item\label{threshold4a}
$B\cap B_F\in \cT_{F}$.
\end{enumerate}
\end{prop}

\begin{proof}
Note that \eqref{threshold3a} $\Leftrightarrow$ \eqref{threshold4a} follows immediately from 
Lemma~\ref{lem:preThreshold}\eqref{pre2}.

To show \eqref{threshold2a} $\Leftrightarrow$ \eqref{threshold4a} note that the 
Addition-Deletion Lemma (Theorem~\ref{thm:generic-fiber-orbi}) implies that
$$F_*:\pi_1(U_{B})\to \pi_1^{\orb}(S\setminus B)$$ 
is an isomorphism if and only if
$$F_*:\pi_1(U_{(B\cap B_F)\cup\{P\}})\to \pi_1^{\orb}\left(S\setminus \left((B\cap B_F)\cup\{P\}\right)\right)$$
is an isomorphism for any $P\in S\setminus (B_F\cup B)$, which is true iff 
$B\cap B_F\in \cT_{F}$.
\end{proof}

In this paper, we will compute the threshold sets of atypical values in cases which fall under the realm of Remark~\ref{rem:non-empty-proj}. In these cases, the threshold 
sets of atypical values contain more information about the fundamental groups of the complements of the 
fiber-type curves that arise from $F$ than the definition of a threshold set of atypical values may suggest, 
as exemplified by the following result, which may be used in combination with Proposition~\ref{prop:BinBF}. 

\begin{prop}\label{prop:threshold-proj}
	Let $F=[f:g]:U=\PP^2\setminus\mathcal B\to \PP^1$ be an admissible map, where $f,g\in\CC[x,y,z]$ are 
	homogeneous polynomials of degree $d\geq 1$ with no non-constant common factors, and $\mathcal B=V(f)\cap V(g)$. 
	Let $B\subset \PP^1$ be such that $\#B\geq 2$. Then, the following are equivalent:
	\begin{enumerate}
\item\label{threshold1} $\pi_1(U_B)$ is a free product of cyclic groups.
\item\label{threshold2} $F$ satisfies \ref{thm:main-simply-1}--\ref{thm:main-simply-3} of Theorem~\ref{thm:main-P2}, 
so in particular, 
$$F_*:\pi_1(U_B)=\pi_1\left(\PP^2\setminus\overline{F^{-1}(B)}\right)\to \pi_1^{\orb}(\PP^1\setminus B)
\quad \text{is an isomorphism.}$$
Moreover, possibly after a change of coordinates in $\PP^1$, $F$ also satisfies Condition~\ref{cond:nice-pencil}.
	\end{enumerate}
\end{prop}

\begin{proof}
The implication \eqref{threshold2} $\Rightarrow$ \eqref{threshold1} is trivial. 
Let us now see the implication \eqref{threshold1} $\Rightarrow$ \eqref{threshold2} 
if $\#B\geq 2$. Assume that $\pi_1(U_B)$ is a free product of cyclic groups. 
By Theorem~\ref{thm:main-P2} there exists a pencil $G=[g_u^v:g_v^u]$ satisfying Condition~\ref{cond:nice-pencil} 
and \ref{thm:main-simply-1}--\ref{thm:main-simply-3} of Theorem~\ref{thm:main-P2}. In particular, every curve of
the form $\overline{F^{-1}(P)}\subset\PP^2$ for $P\in B$ is the union of the closure of finitely many irreducible
fibers of $G$. Let $[\alpha_1:\beta_1],[\alpha_2:\beta_2]$ be distinct points in $B$. Then,
\begin{align*}
\beta_1 f-\alpha_1 g&=g_u^{k_1}g_v^{k_2}\prod_{i=1}^{k_3}(b_ig_u^v-a_ig_v^u)^{l_i},\\
\beta_2 f-\alpha_2 g&=g_u^{k_4}g_v^{k_5}\prod_{j=1}^{k_6}(d_jg_u^v-c_jg_v^u)^{m_j},
\end{align*}
for some $k_j\in\ZZ_{\geq 0}$, $j=1,\ldots, 6$ with $k_1+k_2+k_3\geq 1$, $k_4+k_5+k_6\geq 1$, 
$l_i\in\ZZ_{\geq 1}$ for all $i=1,\ldots, k_3$, $m_j\in\ZZ_{\geq 1}$ for all $j=1,\ldots, k_6$, and 
$[a_i:b_i], [c_j:d_j]$ distinct points in $\PP^1\setminus\{[1:0],[0:1]\}$ for $i=1,\ldots, k_3$ and $j=1,\ldots,k_6$.

Note that, since $G$ satisfies \ref{thm:main-simply-3} of Theorem~\ref{thm:main-P2} and Condition~\ref{cond:nice-pencil}, $b_ig_u^v-a_ig_v^u$ and $d_jg_u^v-c_jg_v^u$ are irreducible polynomials for all $i=1,\ldots, k_3$ and for all $j=1,\ldots, k_6$.  Looking at the degrees of the homogeneous polynomials in the previous two equations, we obtain that
$$
d=k_1u+k_2v+uv\sum_{i=1}^{k_3}l_i=k_4u+k_5v+uv\sum_{j=1}^{k_6}m_i.
$$
Since $u$ and $v$ are coprime, the last equality implies that $v\mid (k_1-k_4)$ and $u\mid (k_2-k_5)$. Note that, since different fibers of $F$ do not have any irreducible components in common, at least one element of the sets $\{k_1,k_4\}$ and $\{k_2,k_5\}$ is $0$, so $v\mid k_1,k_4$ and $v\mid k_2,k_5$. Hence, there exist homogeneous polynomials $h_1,h_2\in\CC[x,y]$ of the same degree $d'$ such that $f=h_1(g_u^v,g_v^u)$ and $g=h_2(g_u^v,g_v^u)$. In other words, $F$ factors as the composition of $G$ and $[h_1:h_2]:\PP^1\to\PP^1$. Since $F$ is admissible, then $d'=1$, and hence $[h_1:h_2]$ is a change of coordinates in $\PP^1$. In particular, $F$ satisfies Condition~\ref{cond:nice-pencil} (after possibly a change of coordinates in $\PP^1$), and also \ref{thm:main-simply-1}--\ref{thm:main-simply-3} of Theorem~\ref{thm:main-P2}, including the fact that
$$
F_*:\pi_1(U_B)\to\pi_1^{\orb}(\PP^1\setminus B)
$$
is an isomorphism.
\end{proof}

In other words, Propositions~\ref{prop:BinBF} and \ref{prop:threshold-proj} show that finding all the threshold 
sets of atypical values for an admissible map $F$ amounts to completely characterizing when the fundamental group of
$U_B=\PP^2\setminus \overline{F^{-1}(B)}$ is a free product of cyclic groups for all finite $B\subset\PP^1$ 
such that $\# B\geq 2$.

\begin{rem}
\label{rem:case1}
	If $\# B=1$ and $B\subset \PP^2\setminus B_F$, the result of Proposition~\ref{prop:threshold-proj} 
	holds by \cite[Theorem~1.2]{ji-Eva-orbifold}.
\end{rem}

However, if $\# B=1$ and $B\subset B_F$, the result is no longer true:

\begin{exam}
	Let $f_2, f_3\in\CC[x,y,z]$ be irreducible homogeneous polynomials of degrees $2$ and $3$ respectively. 
	By Lemma~\ref{lem:irreducibleAdmissible}, $[f_2^3:f_3^2]$ restricted to its maximal domain of definition 
	is an admissible map. By \cite[Theorem 1.2]{ji-Eva-orbifold}, 
	$\pi_1(\PP^2\setminus V(f_2^3+\lambda f_3^2))\cong\ZZ_2*\ZZ_3$ for $\lambda\in\CC$ generic. In particular, 
	$f_2^3+\lambda f_3^2$ is an irreducible degree $6$ homogeneous polynomial. Let $f_5\in \CC[x,y,z]$ be an 
	irreducible degree $5$ homogeneous polynomial. By Lemma~\ref{lem:irreducibleAdmissible}, the restriction 
	of the pencil $F=[(f_2^3+\lambda f_3^2)^5:f_5^6]$ to its maximal domain of definition is an admissible map. 
	For this $F$, if $B=\{[0:1]\}$ we have that $\pi_1(U_B)=\ZZ_2*\ZZ_3$ is a free product of cyclic groups, but
	$$
	F_*:\pi_1(U_B)\to\pi_1^{\orb}(\PP^1\setminus B)=\ZZ_6
	$$
	is not an isomorphism.
\end{exam}

\section{Zariski pairs of fiber-type curves}\label{sec:zar-pairs}
Consider $F:\PP^2\dashrightarrow \PP^1$ the rational morphism defined as $[\bar f(x,y,z):z^6]$, 
where $f(x,y)=x^6+y^6+6xy$ and $\bar f(x,y,z)=x^6+y^6+6xyz^4$. 
This rational map is not well defined on the base points $\mathcal B=\{[x:y:0]\in \PP^2 \mid x^6+y^6=0\}$. 
The generic fiber of the resulting admissible map $F:U\to\PP^1$, where $U=\PP^2\setminus\mathcal{B}$,
is a smooth sextic (with six points removed). The set of atypical values of $F$ is 
$B_{F}=\{0,\infty,\pm 4,\pm 4 \cI\}\subset\PP^1\equiv\CC\cup\{\infty\}$. The fiber $F^{-1}(\infty)$
is the only multiple fiber and it has multiplicity six. The fiber $F^{-1}(0)$ is a singular sextic with one node. 
The remaining atypical fibers are irreducible sextics with six nodes. In particular, $F$ is an admissible map satisfying Condition~\ref{cond:nice-pencil}. 
For any finite set $B\subset \PP^1$ we will denote $C_B:=\overline{F^{-1}(B)}$. 
The purpose of this section is to study $\pi_1(U_B)$ for $U_B:=\PP^2\setminus C_B$ for different choices
of~$B$. In particular, we will exhibit a Zariski pair of affine nodal curves and another Zariski pair of projective 
sextics of fiber-type. 

\subsection{A Zariski pair of affine nodal curves}\label{sec:zar-pair}
Using the notation presented above, we consider the curve $C_{B_1}$ (resp. $C_{B_2}$) for the set 
$B_1=\{4,-4,\infty\}$ (resp. $B_2=\{4, 4\cI,\infty\}$). Using the identification 
$\PP^2\setminus\{z=0\}=\PP^2\setminus \overline{C_\infty}\equiv\CC^2$ we can consider $U_{B_j}$ as 
the complement in $\CC^2$ of the affine curves~$D_j$ defined as a union of the two fibers over 
$B_j\setminus\{\infty\}$ of the polynomial map $(x,y)\mapsto f(x,y)=x^6+y^6+6xy$, for $j=1,2$. In this context one has the following.

\begin{thm}
\label{thm:Zariskipair}
The curves $(D_1,D_2)$ described above form a Zariski pair of affine nodal curves of degree 12. 
Moreover, $\pi_1(U_{B_2})=\FF_2$ whereas 
\begin{equation}
\label{eq:presentation1}
\pi_1(U_{B_1})=\langle x,y,u,v: [u,x]=[v,y]=1, u=[v,x], v=[u,y]\rangle
\end{equation}
fits in the following free-by-free extension 
\begin{equation}
\label{eq:presentation2}
1\to \ZZ u * \ZZ v \to \pi_1(U_{B_1})\to \ZZ x * \ZZ y \to 1
\end{equation}
and is not free.
\end{thm}

\begin{proof}
Since the coefficients of the equations of the curves can be given in a number field such as $\QQ(\cI)$, 
one can obtain presentations of the fundamental groups accurately using the package \textit{Sirocco}
in the computer software SageMath\footnote{A program to do so is available in the public repository:\\
\texttt{https://riemann.unizar.es/\~{}jicogo/software/AffineZariskiPair.ipynb} (Jupyter version),\\
\texttt{https://riemann.unizar.es/\~{}jicogo/software/AffineZariskiPair.sage} (plain text version).}
\cite{Marco-sirocco}.

In order to give a comprehensive proof of this result we will perform a series of transformations to $D_j$ 
in order to simplify it and be able to calculate its braid monodromy and apply Zariski-Van Kampen's method.
Consider the double cover $\CC^2\ \rightmap{\sigma}\ \CC^2$ of symmetries defined as $\sigma(x,y)=(xy,x+y)=(u,v)$. 
The discriminant of $\sigma$ is given by $\Delta_\sigma(u,v)=v^2-4u$ and $\sigma_|$ defines an unramified double 
cover outside $\sigma^{-1}(\Delta_{\sigma})=x-y=0$. Note that $f(x,y)=g(\sigma(x,y))$, where
$g(u,v)=v^6-6v^4u+9v^2u^2-2u^3+6u$.
Consider $\CC^2\ \rightmap{\delta}\ \CC^2$ the cover $\delta(u,v)=(u,v^2)$ associated with the reflection 
over $\ell_1=\{v=0\}$, which is an unramified double cover outside of the preimage of $\ell_1$. One obtains 
that $f(x,y)=(h\circ\delta\circ\sigma)(x,y)$, where $h(u,v)=v^3-6v^2u+9vu^2-2u^3+6u$. In order to compute
$G_j=\pi_1(U_{B_j})$, we will first calculate 
$\tilde G_j=\pi_1(\CC^2\setminus (H_{B_j\setminus\{\infty\}}\cup \ell_1\cup \ell_2))$, 
where $H_B=h^{-1}(B)$ and $\ell_2=\{v=4u\}$. Once $\tilde G_j$ is computed one can obtain $G_j$ after 
a sequence of index 2 subgroups associated with the double covers and 
quotients by meridians in order to eliminate the ramification locus.

Note that the polynomial map $h$ produces an elliptic fibration whose generic fiber is a smooth cubic. 
This fibration is ramified at $B_h=\{\pm 4,\pm 4 \cI\}$ whose fibers are of type $I_1$ (nodal cubics).
In addition, the line $\ell_1$ (resp. $\ell_2$) is tangent to the fibers $H_B$, for $B=\{0,\pm 4\}$ 
(resp. $B=\{\pm 4\cI\}$). 
This implies that $f^{-1}(\lambda)$, $\lambda\in B_h$ contains 4 nodes coming from the preimages of the nodes 
at $H_\lambda$ plus two nodes coming from the tangency at $\ell_1 \cap H_\lambda$ (resp. $\ell_2 \cap H_\lambda$).
Finally, the fiber $f^{-1}(0)$ contains only one node since the tangency $\ell_1 \cap H_\lambda$ produces a node 
after the preimage of $\delta$, which is on the ramification locus of $\sigma$ and hence produces only one node 
after the preimage by~$\sigma$.

\begin{center}
	\begin{figure}[ht]
		\includegraphics[scale=.4]{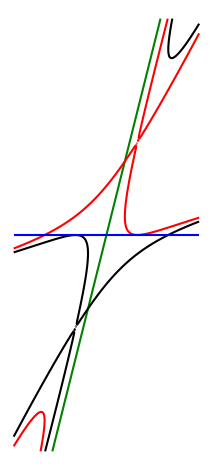}
		\caption{Real curve $H_{B_1\setminus\{\infty\}}\cup \ell_1\cup\ell_2$}
		\label{fig:curvareal}
	\end{figure}
\end{center}
Using the real picture (see Figure~\ref{fig:curvareal}) one can obtain a braid monodromy for 
$H_{B_1\setminus\{\infty\}}\cup \ell_1\cup\ell_2$ 
and use Zariski-Van Kampen's method to obtain a presentation of the fundamental group 
$\tilde G_1=\pi_1(\CC^2\setminus (H_{B_1\setminus\{\infty\}}\cup \ell_1\cup \ell_2))$ as 
\begin{equation}
\label{eq:tildeG1}
\tilde G_1=\left\langle \gamma_1,\gamma_2,\gamma_+,\gamma_-: 
\array{l}
[\gamma_1,\gamma_2]=[\gamma_2,\gamma_\pm]=[\gamma^2_1,\gamma_\pm]=1\\
(\gamma_1\gamma_+)^2=(\gamma_+\gamma_1)^2, (\gamma_1\gamma_-)^2=(\gamma_-\gamma_1)^2\\
{[\gamma_-,\gamma_+\gamma_1\gamma_+^{-1}]=[\gamma_+,\gamma_-^{-1}\gamma_1\gamma_-]=1}\\
\endarray
\right\rangle,
\end{equation}
where $\gamma_j$ is a meridian of the line $\ell_j$ ($i=1,2$) and $\gamma_\pm$ is a meridian of $H_{\pm 4}$
on the vertical line $x=\varepsilon$. The fundamental group of the double cover $\delta$ is obtained as the
kernel $\tilde K_1$ of the morphism $\tilde G_1\to \ZZ_2$ defined as $\gamma_1 \mapsto 1$, $\gamma_2 \mapsto 0$,
$\gamma_\pm \mapsto 0$.

The group $\tilde K_1$ is generated by 
$\gamma_1^2,\gamma_2,\tilde \gamma_2:=\gamma_1\gamma_2\gamma_1^{-1},\gamma_\pm,\tilde \gamma_\pm:=\gamma_1\gamma_\pm\gamma_1^{-1}$
and has relations
$$
\array{ll}
\tilde \gamma_2=\gamma_2\gamma_1^2=\gamma_1^2\gamma_2,&
{[\gamma_\pm,\gamma_2]=[\tilde\gamma_\pm,\gamma_2]=1,}\\
{[\gamma_\pm,\gamma_1^2]=[\tilde\gamma_\pm,\gamma_1^2]=1,}&
{[\gamma_\pm,\tilde\gamma_\pm]=1,}\\
\gamma_1^2\tilde\gamma_+\gamma_-\gamma_+=\gamma_+\tilde\gamma_-\tilde\gamma_+&
\gamma_1^2\gamma_-\gamma_+\tilde\gamma_-=\tilde\gamma_-\tilde\gamma_+\gamma_-.\\
\endarray
$$
Since $\gamma_1^2$ is the meridian of the conic $\delta^*(\ell_1)$, the quotient 
$$
\tilde K_1/\langle \gamma_1^2\rangle=\ZZ\gamma_2\times \left\langle \gamma_\pm,\tilde\gamma_\pm: 
\array{l}
[\gamma_\pm,\tilde\gamma_\pm]=1,\\
\tilde\gamma_+\gamma_-\gamma_+=\gamma_+\tilde\gamma_-\tilde\gamma_+,
\gamma_-\gamma_+\tilde\gamma_-=\tilde\gamma_-\tilde\gamma_+\gamma_-
\endarray
\right\rangle.
$$
is the fundamental group of the complement $Z_1$ of $g^{-1}(B_1\setminus\{\infty\})\cup \ell_2$. 
Now, the fundamental group $K_1$ of $\sigma^{-1}(Z_1)$ is isomorphic to the kernel of the homomorphism 
$\tilde K_1/\langle \gamma_1^2\rangle \to \ZZ_2$ defined by $\gamma_2 \mapsto 1$, $\gamma_\pm \mapsto 0$,
$\tilde\gamma_\pm \mapsto 0$, that is, 
$$
K_1=\ZZ\gamma_2^2\times \left\langle \gamma_\pm,\tilde\gamma_\pm: 
\array{l}
[\gamma_\pm,\tilde\gamma_\pm]=1,\\
\tilde\gamma_+\gamma_-\gamma_+=\gamma_+\tilde\gamma_-\tilde\gamma_+,
\gamma_-\gamma_+\tilde\gamma_-=\tilde\gamma_-\tilde\gamma_+\gamma_-
\endarray
\right\rangle.
$$ 
After factoring out the normal subgroup generated by $\gamma_2^2$, the meridian of $\sigma^*(\delta^*(\ell_2))$,
one obtains the desired presentation in~\eqref{eq:presentation1} where $x:=\gamma_-$, $y:=\gamma_-\gamma_+\gamma_-^{-1}$,
$u:=\tilde\gamma_+\gamma_+^{-1}$, and $v:=\gamma_+^{-1}\tilde\gamma_-\gamma_-^{-1}\gamma_+$.

Now consider the case $B_2=\{4,4\cI,\infty\}$. The strategy to follow is similar to that of the previous case.
However, note that the curve $H_{B_2\setminus\{\infty\}}$ is not real anymore. A way around this is to consider 
$h(u\cI,v\cI)-4\cI=-\cI\left(\tilde h(u,v)-4\right)$, where $\tilde h(u,v)=v^3-6v^2u+9vu^2-2u^3-6u$.
This allows us to consider the braid monodromy along real loops (for $H_4, \ell_1, \ell_2$) and purely imaginary 
loops for $H_{4\cI}$ as a braid monodromy over a real path for $\tilde H_4=\tilde h^{-1}(4)$. This produces the 
following presentation,
\begin{equation}
\label{eq:tildeG2}
\tilde G_2=\pi_1(\CC^2\setminus (H_{B_2\setminus\{\infty\}}\cup \ell_1\cup \ell_2))=
\ZZ \gamma_1\times \ZZ \gamma_2\times \langle \gamma_+,\gamma_-\rangle\cong \ZZ^2\times\FF_2,
\end{equation}
where $\gamma_j$ is a meridian of the line $\ell_j$ ($j=1,2$) and $\gamma_+$ (resp. $\gamma_-$) is a meridian 
of $H_{4}$ (resp. $H_{4\cI}$). Following the same strategy as above 
$\tilde K_2\cong\ZZ\gamma_1^2\times\ZZ\gamma_2\times\FF_2$, 
so $\tilde K_2/\langle\gamma_1^2\rangle\cong \ZZ\gamma_2\times\FF_2$. 
Consequently $K_2\cong \ZZ\gamma_2\times\FF_2$, and $G_2\cong \FF_2$.  

To prove the short exact sequence~\eqref{eq:presentation2}, note that the kernel of the projection map 
$x\mapsto x$, $y\mapsto y$, $u\mapsto 1$, $v\mapsto 1$ 
is generated by $u_w:=wuw^{-1}$ and $v_w:=wvw^{-1}$ for any $w\in \ZZ x * \ZZ y$. From the 
presentation~\eqref{eq:presentation1} of $G_1$, one obtains the following recursive set of relations 
$$
\array{llll}
u_{wx}=u_{wx^{-1}}=u_{w}, & v_{wy}=v_{wy^{-1}}=v_{w},\\
u_{wy}=v_{w}^{-1}u_w, & v_{wx}=u_{w}^{-1}v_w,\\
u_{wy^{-1}}=v_wu_w, & v_{wx^{-1}}=u_{w}v_w.
\endarray
$$
This allows one to reduce the set of generators of the kernel to $u_1=u$ and $v_1=v$ and no further relations. 
Since the abelianization of $G_1$ has rank $2$ and finitely generated free groups are Hopfian, the short 
exact sequence~\eqref{eq:presentation2} shows that $G_1$ cannot be free.
\end{proof}

\begin{rem}
Note that both spaces $U_{B_1}$ and $U_{B_2}$ are associated with affine pencils without multiple fibers and hence 
the epimorphisms of their fundamental groups onto the free group $\FF_2$ must have finitely generated kernels 
as shown by Catanese in~\cite[Thm. 5.4]{Catanese-Fibred}.
\end{rem}

The following result shows that the nodes of the curve $D_1$ are in special position, unlike those in~$D_2$.

\begin{lemma}
	\label{lemma:nodes}
	The 12 nodes in the curve $D_1$ are contained in three lines containing $4$ nodes each.
	There are no three nodes aligned in~$D_2$.
\end{lemma}

\begin{proof}
	We use the notation introduced in the proof of Theorem~\ref{thm:Zariskipair}. It is easy to check that the node of $H_{\pm 4}$ (resp. $H_{\pm 4\cI}$) has coordinates $P_{\pm 4}=(u,v)=(\pm 1,\pm 3)$
	(resp. $P_{\pm 4\cI}=(u,v)=(\pm \frac{\cI}{4},\pm \frac{\cI}{4})$). Also, $H_{\pm 4}$ (resp. $H_{\pm 4\cI}$)
	is tangent to the line $\ell_1=\{v=0\}$ (resp. $\ell_2=\{v=4u\}$) at the point $Q_{\pm 4}=(\pm 1,0)$
	(resp. $Q_{\pm 4\cI}=(\pm \cI,\pm 4\cI)$).
	The line $\ell_+=\{v=3u\}$  joining $P_{4}$ and $P_{-4}$ is transformed through $\delta\circ\sigma$ into the
	singular conic $Q_+=\{3x^2+5xy+3y^2=0\}$ containing the 8 nodes which are in the preimage of $P_{\pm 4}$.
	Each of the lines in the conic $Q_+$ contains 4 of those nodes. In addition, the line $\ell_1$ is transformed
	through $\sigma$ into the line $x+y$ which contains the remaining 4 nodes which are the preimage of the
	tangencies~$Q_{\pm 4}$.
	
	On the other hand, let $\ell_-=\{(4-\cI)v-(12-\cI)u+2\cI=0\}$ be the line joining $P_{4}$ and $P_{4\cI}$.
	The preimage of $\ell_-$ through $\delta\circ\sigma$ is a smooth conic $Q_-$ containing the 8 nodes
	preimage of $P_{4}$ and $P_{4\cI}$. This proves that no three of those preimages can be aligned.
	An analogous argument rules out the existence of a line joining three points in the preimage of any two of the
	points of~$S=\{P_{4},P_{4\cI},Q_{4},Q_{4\cI}\}$.
	
	Assume that there is a line joining a preimage of $P_{4}$, $P_{4\cI}$, and $Q_{4}$, and consider the quartic
	consisting on the four lines obtained from it by applying the deck transformations of the global 4:1 covering
	induced by the restriction of $\delta\circ\sigma$. This invariant quartic is the preimage by $\delta\circ\sigma$
	is a conic which is tangent to $\ell_1$ at $Q_4$, since the preimages of $Q_4$ are double points in the quartic.
	However, this complete reducible quartic has four extra double points. Since they cannot be on the ramification
	locus, they have to be on a fiber. Hence, the conic has to be singular. This contradicts the condition about the
	tangency with $\ell_1$.
	An analogous argument rules out the existence of a line joining any three points in the preimage of three of the
	points of~$S$.
\end{proof}

Let $\omega:=e^{\frac{1}{12}\pi\cI}$. 
Note that the isomorphism
\begin{equation}
\label{eq:transfproj}
\array{rrcl}
\underline{\omega}: & \PP^2 & \to & \PP^2\\
& [x:y:z] & \mapsto & [\omega x:\omega^5 y:z],
\endarray
\end{equation}
takes $\overline{F^{-1}([a:b])}$ to $\overline{F^{-1}([a\cI:b])}$ 
for all $[a:b]\in \PP^1$. Using this transformation, Theorem~\ref{thm:Zariskipair}
immediately implies the following.

\begin{cor}\label{cor:non-free}
$\pi_1(U_{\{4,-4,\infty\}})\cong \pi_1(U_{\{4\cI,-4\cI, \infty\}})$ is the non-free group with presentation~\eqref{eq:presentation1}.
\end{cor}

We end this section by finding more affine fiber-type curves $D\subset\CC^2$ whose irreducible components are atypical fibers of $f$ and such $\pi_1(\CC^2\setminus D)$ is also free.

\begin{prop}
\label{prop:threshold}
Let $\underline\omega:\PP^2\to\PP^2$ be the isomorphism from equation~\eqref{eq:transfproj}. Then, 
\label{prop:free}
$$
\FF_3\cong\pi_1(U_{\{0,4,4\cI,\infty\}})\ \congmap{\underline{\omega}}\ \pi_1(U_{\{0,-4,4\cI,\infty\}})
\ \congmap{\underline{\omega}}\ \pi_1(U_{\{0,-4,-4\cI,\infty\}})\ \congmap{\underline{\omega}}\ 
\pi_1(U_{\{0,4,-4\cI,\infty\}}).$$
\end{prop}

\begin{proof}
Similar, but more involved calculations as those detailed in the proof of Theorem~\ref{thm:Zariskipair}, 
show that $G_0=\pi_1(U_{B_0})=\FF_3$ for $B_0=\{0,4,4\cI,\infty\}$.
In fact, using the notation from the proof of Theorem~\ref{thm:Zariskipair}, we have that
$$\tilde G_0=\pi_1(\CC^2\setminus (H_{B_0\setminus\{\infty\}}\cup\ell_1\cup\ell_2))=
\ZZ\gamma_1\times\ZZ\gamma_2\times\langle\gamma_0,\gamma_+,\gamma_-\rangle\cong\ZZ^2\times\FF_3,
$$
where $\gamma_1$ (resp. $\gamma_2$) are meridians of the lines $\ell_1$ (resp. $\ell_2$), $\gamma_0$ is a 
meridian around $H_0$ and $\gamma_+$ (resp. $\gamma_-$) is a meridian of $H_4$ (resp. $H_{4\cI}$).
\end{proof}

\subsection{A projective Zariski pair}\label{sec:projective-ZP}
In this section we continue using the notation presented at the beginning of~\S\ref{sec:zar-pairs}.
One can obtain a Zariski pair of projective curves in $\PP^2$ as follows.

\begin{thm}\label{thm:projPair}
The projective curves $(C_{\{4,-4\}},C_{\{4,4\cI\}})$ form a Zariski pair of projective curves of degree 12 
whose complements in $\PP^2$ have non-isomorphic fundamental groups.
More concretely, $\pi_1(U_{\{4,4\cI\}})=\ZZ*\ZZ_6$, whereas
\begin{equation}
\label{eq:proj-presentation}
\pi_1(U_{\{4,-4\}})\cong\pi_1(U_{\{4,-4,\infty\}})/\langle (u(yx)^3)^2\rangle,
\end{equation}
where $\langle \gamma\rangle$ denotes the normal subgroup of $\pi_1(U_{\{4,-4,\infty\}})$ generated by~$\gamma$.
\end{thm}

\begin{proof}
Let us denote $B_1'=\{4,-4\}$ and $B_2'=\{4,4\cI\}$.
The proof of $\pi_1(U_{B_2'})=\ZZ*\ZZ_6$ is a consequence of Theorem~\ref{thm:Zariskipair} and 
the Deletion Lemma in Theorem~\ref{thm:generic-fiber-orbi}.
The presentation of $\pi_1(U_{B_1'})$ can be obtained from the presentation of $\pi_1(U_{B_1})$ by quotienting 
by the normal subgroup generated by the meridian $\tilde\gamma_{\infty}$ around the line at 
infinity $\ell_{\infty}$. This meridian can be obtained by lifting the square $\gamma^2_{\infty}$ of the meridian 
around the line at infinity by the coverings $\sigma$ and $\delta$ from the proof of Theorem~\ref{thm:Zariskipair}. 
The real picture given in Figure~\ref{fig:curvareal} shows that 
$$\gamma^{-1}_{\infty}=\gamma_2\gamma_1(\gamma_+\gamma_-)^2\gamma_1^{-1}\gamma_-\gamma_1\gamma_-^{-1}\gamma_+\gamma_-.$$
Following the transformations given by the covering, one obtains
$$
\tilde\gamma_{\infty}^{-1}=
(\tilde\gamma_-(\gamma_+\gamma_-)^2\gamma_+)^2=(u(yx)^3)^2.
$$
We will show that $\pi_1(U_{B_1'})\not\cong \ZZ*\ZZ_6$. Let $N$ be the normal closure of the subgroup generated by $u$ and $v$. Note that $\pi_1(U_{B_1'})/N=\langle x,y\, :\, (yx)^6\rangle\cong \ZZ*\ZZ_6$. Hence, since $\ZZ*\ZZ_6$ is Hopfian, $\pi_1(U_{B_1'})\cong \ZZ*\ZZ_6$ if and only if $N$ is the trivial subgroup, which happens if and only if both $u$ and $v$ are trivial $\pi_1(U_{B_1'})$. Using the presentation of $\pi_1(U_{B_1'})$ it is easy to see that the rule
$$
\begin{array}{cccc}
\varphi: & \pi_1(U_{B_1'})&\longrightarrow & S_4\\
 & x &\longmapsto & (34)\\
  & y &\longmapsto & (24)\\
   & u &\longmapsto & (12)(34)\\
    & v &\longmapsto & (13)(24)\\
\end{array}
$$
defines a group homomorphism, and since the images of $u$ and $v$ are non-trivial elements of $S_4$,
neither $u$ nor $v$ are trivial in $\pi_1(U_{B_1'})$.
\end{proof}

\begin{rem}\label{rem:position}
As seen in the proof of Lemma~\ref{lemma:nodes}, the three lines containing the 12 nodes in $C_{\{4,-4\}}$ 
are the preimages of $\ell_1=\{v=0\}$ and $\ell_+=\{v=3u\}$. The latter passes through the nodes $(\pm 1,\pm 3)$
and it intersects the two cubics $H_{\pm 4}$ transversally at two other smooth points, namely $(\mp 2,\mp 6)$, 
and hence the preimage of $\ell_+$ by $\delta\circ\sigma$ (which splits in two lines $6x+(5+\sqrt{11}\cI)y=0$ and 
$6x+(5-\sqrt{11}\cI)y=0$) contains 8 nodes (4 each line) and it intersects 
$C_{\{4,-4\}}$ transversally at the remaining 8 points of intersection (again, 4 each line). On the other hand,
$\ell_1$ is tangent to the two cubics at $(\pm 1,0)$ and it intersects them transversally at two other points
$(\mp 2,0)$. This implies that its preimage (a ramification line $x+y=0$) passes through 4 nodes (the preimage
of a simple tangency to a ramification locus of order two becomes two nodes) and through 4 additional transversal
intersections (the preimages of $(\pm 2,0)$). In particular, the three lines passing through the 12 nodes are 
concurrent and the remaining intersection points of any of these three lines with the curve $C_{\{4,-4\}}$ are 
transversal intersections (4 transversal intersections per line). This stresses the fact that the only difference 
between the position of the singularities of the curves 
$(C_{\{4,-4\}},C_{\{4,4\cI\}})$ is the contribution of the special position of the nodes, and the same goes for 
$(C_{B_1}=C_{\{4,-4,\infty\}},C_{B_2}=C_{\{4,4\cI,\infty\}})$ (that is, the projective curves whose complements 
in $\PP^2$ are $\CC^2\setminus D_1$ and $\CC^2\setminus D_2$ respectively).
\end{rem}

\subsection{Threshold sets of atypical values for~$F$}\label{sec:threshold-example}
\mbox{} 

The following result shows that all the threshold sets of atypical values for $F$
contain only two elements from $\{4,-4,4\cI,-4\cI\}$, but not all choices are allowed.

\begin{prop}\label{prop:threshold-generic}
There are four threshold sets of atypical values for $F$, namely $T\in\cT_F$ is maximal if and only if
$T=\{0,\infty,\varepsilon_1 4,\varepsilon_2 4\cI\}$, where $\varepsilon_1,\varepsilon_2\in\{\pm 1\}$.
\end{prop}
\begin{proof}
By Theorem~\ref{thm:projPair}, $F_*:\pi_1(U_{\{4,-4\}})\to \pi_1^{\orb}(\PP^1\setminus \{4,-4\})\cong\ZZ*\ZZ_6$ 
is not an isomorphism. By Propositions~\ref{prop:BinBF} and \ref{prop:threshold-proj} one has  
$\{4,-4\}\notin \cT_F$. Using the transformation $\underline{\omega}$
from~\eqref{eq:transfproj} we also obtain that $\{4\cI,-4\cI\}\notin \cT_F$.
Hence, by Lemma~\ref{lem:preThreshold}\eqref{pre2}, no subset of $B_F$ containing $\{4,-4\}$ or $\{4\cI,-4\cI\}$
can be in $\cT_F$. The previous discussion and Proposition~\ref{prop:threshold} imply the result.
\end{proof}

Finally, we explicitly state the consequence of the computation of threshold sets.

\begin{thm}\label{thm:allPairs}
Let $F:U\to\PP^1$ be given by $F=[x^6+y^6+6xyz^4:z^6]$, where $U=\PP^2\setminus\{[x:y:0]\mid x^6+y^6=0\}$. 
Let $B_F\subset\PP^1$ be the set of atypical values of $F$. For every finite set $A\subset\PP^1$, let 
$C_A=\overline{F^{-1}(A)}\subset\PP^2$. 

Then, the following is a complete list of all Zariski pairs of curves $(C_A,C_{A'})$ such that 
the fundamental group of the complement of $C_A$ in $\PP^2$ is a free product of cyclic groups:

\vspace*{7pt}
\begin{center}
\begin{tabular}{|c|c|c|}
$\pi_1(\PP^2\setminus C_{A})$ & $A$ & $A'$ \\
\hline
$\FF_{r+1}*\ZZ_6$ & $\{\varepsilon_14,\varepsilon_24 \cI\}\cup B$ & $\{4\alpha,-4\alpha\}\cup B'$\\
\hline
$\FF_{r+2}$ & $\{\varepsilon_14,\varepsilon_24\cI,\infty\}\cup B$ & $\{4\alpha,-4\alpha,\infty\}\cup B'$\\
\hline
$\FF_{r+2}*\ZZ_6$ & $\{0, \varepsilon_14,\varepsilon_24\cI\}\cup B$ & $\{0,4\alpha,-4\alpha\}\cup B'$\\
\hline
$\FF_{r+3}$ & $\{0,\varepsilon_14,\varepsilon_24\cI,\infty\}\cup B$ & $\{0,4\alpha,-4\alpha,\infty\}\cup B'$\\
\hline
\end{tabular}
\end{center}
\vspace*{7pt}
where $B,B'\subset\PP^1\setminus B_F$ are finite, $r:=\# B= \# B'$, 
$\varepsilon_1,\varepsilon_2\in\{\pm 1\}$, and $\alpha\in\{1,\cI\}$.
\vspace*{7pt}
\end{thm}

\begin{proof}
Recall from the beginning of the section, that $B_F=\{0,\infty,\pm 4,\pm 4\cI\}$.
Let's write $A=A_1\cup B$, where $A_1=A\cap B_F$ and $B=A\setminus B_F$ 
(and analogously we define $A_1'$ and $B'$ from $A'$).
Since the closure of every fiber of $F$ is an irreducible curve in $\PP^2$, Zariski pairs of the form 
$(C_A,C_{A'})$ must satisfy that $\# A_1=\# A'_1$, $\# B=\# B'$, and the combinatorial type of the fibers at 
$A_1$ are in one-to-one correspondence with those at $A'_1$.
Also note that, if $A_1=A'_1$ one can use an isotopy of the base that fixes $B_F$ and takes $B$ to $B'$ to 
show that $(\PP^2,C_A)$ is isotopic to $(\PP^2,C_{A'})$ and hence $(\PP^2,C_A)\cong (\PP^2,C_{A'})$. 

Let $(C_A,C_{A'})$ be a Zariski pair, such that $\pi_1(\PP^2\setminus C_A)$ is a free product of cyclic groups.
We will prove that $A$ (resp. $A'$) must be as in the middle (right-most) column in the table above for one row.

Assume $\# (A_1\cap \{\pm 4,\pm 4\cI\}) =\# (A_1'\cap \{\pm 4,\pm 4\cI\})\leq 1$. Since the transformation 
$\underline{\omega}:\PP^2\to \PP^2$ from equation~\eqref{eq:transfproj} acts transitively 
on the fibers of $F$ at $\{\pm 4,\pm 4\cI\}$ and fixes the fibers at $\{0,\infty\}$, up to a power of 
$\underline{\omega}$ one can assume $A_1=A'_1$. By the first paragraph, $(\PP^2,C_A)\cong (\PP^2,C_{A'})$, 
so $(C_A,C_{A'})$ cannot be a Zariski pair.

Therefore, $\# (A_1\cap \{\pm 4,\pm 4\cI\}) =\# (A_1'\cap \{\pm 4,\pm 4\cI\})\geq 2$. 
This means in particular $\# A_1=\# A'_1\geq 2$ and hence 
Propositions~\ref{prop:BinBF} and \ref{prop:threshold-proj}
imply $\pi_1(\PP^2\setminus C_A)$ is a free product of cyclic groups if and only if $B_F\setminus A\in\cT_F$.
By Proposition~\ref{prop:threshold-generic}, the only possible $A$ satisfying the hypothesis are shown on the 
middle column of the table above. 
Note that $A'$ could be either a set from the middle column on the table above (shown as $A$) 
or a set from the right-most column. Since $\# A'_1\geq 2$, Proposition~\ref{prop:threshold-proj} implies that 
$\pi_1(\PP^2\setminus C_{A'_1})$ is a free product of cyclic groups in the former case, whereas it is not
in the latter. This shows that all cases on the table are Zariski pairs that can be distinguished by the 
groups of their complements.

Finally, if $\pi_1(\PP^2\setminus C_{A'})$ and $\pi_1(\PP^2\setminus C_{A})$ are isomorphic to a free product 
of cyclic groups (that is, both $A_1$ and $A'_1$ are shown in the middle column on the table above), then one 
can use $\underline{\omega}$ and assume $A_1=A'_1$. Again, as above, $(\PP^2,C_A)\cong(\PP^2,C_{A'})$. 
\end{proof}

\section{Twisted Alexander polynomials for $\SU(2)$-representations}\label{sec:tw-alex}
The purpose of this section is to show that twisted Alexander polynomials can also distinguish the fundamental
groups $G_1=\pi_1(U_{B_1})$ and $G_2=\pi_1(U_{B_2})$ presented in~\ref{sec:zar-pair}. To our knowledge, this
is the first example in the literature of a Zariski pair whose groups have the same Alexander invariant but
different twisted Alexander polynomials. Moreover, Lemma~\ref{lemma:nodes} shows that the position of the nodes
affects the fundamental groups and twisted Alexander polynomials are sensitive to this special position.
Classical Alexander polynomials and characteristic varieties are known not to be sensitive to the position of
nodes (see for instance~\cite{Libgober-Alexander-invariants} and \cite[2.5]{Libgober-characteristic}). We refer
the reader to \cite[Section 2]{CF} for a quick introduction on twisted Alexander polynomials and how to compute them.

We first compute the Alexander invariant of the complements $U_{B_1}$ and $U_{B_2}$.

\begin{prop}
	\label{prop:AlexInv}
	The Alexander invariant of both $\CC^2\setminus D_1$ and $\CC^2\setminus D_2$ is $\CC[t_1^{\pm 1},t_2^{\pm 1}]$.
\end{prop}

\begin{proof}
	Let us denote by $M_j$ the Alexander invariant of $\CC^2\setminus D_j$, that is, $M_j:=G'_j/G''_j$ as an $R$-module where
	$R=\CC[G_j/G'_j]=\CC[t_1^{\pm 1},t_2^{\pm 1}]$. It is immediate that $M_2=[x,y]R$.
	As for $M_1$, one can use the presentation~\eqref{eq:presentation1}, where the abelian classes of $x$ and $y$
	generate $G_1/G'_1$ and $u,v\in G'_1$.
	Using \cite[\S 2.5]{Artal-ji-Tokunaga-survey-zariski}, this implies that $M_1$ is generated by $[x,y]$, $u$, and $v$
	and whose relations are given by rewriting the relations in~\eqref{eq:presentation2} as elements of the $R$-module $M_1$.
	Since $[x,q]=0$ in $M_1$ as long as $q\in G'_1$, the first two relations in~\eqref{eq:presentation2} become trivial
	and the last ones become $u=v=0$. This shows $M_1\cong M_2\cong R$.
\end{proof}

By Proposition~\ref{prop:AlexInv}, the characteristic varieties of $U_{B_1}$ and $U_{B_2}$ are both $(\CC^*)^2$.
Analogously, the Alexander polynomials associated with any homomorphism $\varepsilon:G_j\to \ZZ^2$ are trivial.

In light of Lemma~\ref{lemma:nodes}, this is the announced result showing that twisted Alexander polynomials
(associated with non-abelian representations) can be sensitive to the position of nodes.

\begin{prop}
	The twisted Alexander polynomial for unitary representations in $\SU(2)$
	distinguishes $\pi_1(U_{B_1})$ and $\pi_1(U_{B_2})$.
\end{prop}

\begin{proof}
	By definition, note that if $G=\FF_r$ is a free group, then $\Delta_{\rho,\varepsilon}(t)=1$
	for any representation $\rho$ of $G$ and any homomorphism $\varepsilon:G/G'\to \ZZ$.
	In particular, $\Delta_{\rho_2,\varepsilon_2}(t)=1$ for any twisted Alexander polynomial of
	the free group $G_2$. Hence, it is enough to find a representation, say
	$\rho_1:G_1\to \SU(2)$ and a homomorphism $\varepsilon_1:G_1/G'_1\to\ZZ$ such that
	$\Delta_{\rho_1,\varepsilon_1}(t)\neq 1$. Consider $\varepsilon_1:G_1/G'_1\to\ZZ$ defined
	as $\varepsilon_1(x)=1$, $\varepsilon_1(y)=-1$, $\varepsilon_1(u)=\varepsilon_1(v)=0$ and the
	irreducible representation~$\rho_1:G_1\to \SU(2)$ defined as:
	$$
	\array{ll}
	\rho_1(x):=
	\left(
	\array{cc}
	\xi & 0\\
	0 & \bar\xi
	\endarray
	\right),
	&
	\rho_1(y):=
	-\frac{\sqrt{2}}{2}\left(
	\array{cc}
	1 & 1\\
	-1 & 1
	\endarray
	\right),
	\\
	\rho_1(u):=
	\cI
	\left(
	\array{cc}
	-1 & 0\\
	0 & 1
	\endarray
	\right),
	&
	\rho_1(v):=
	\left(
	\array{cc}
	0 & -1\\
	1 & 0
	\endarray
	\right),
	\endarray
	$$
	where $\xi=\frac{\sqrt 2}{2}(1+i)$ is a primitive 8th root of unity.
	One can check that $A_1$ below is equivalent to the Fox matrix for the
	relations in~\eqref{eq:presentation1} with respect to $\rho_1$ and $\varepsilon_1$.
	$$
	A_1=
	{\small{
			\left(\begin{array}{c|rr|rr}
				* & 0 & 0 & (t-\bar\xi) & 0 \\
				& 0 & 0 & 0 & (t - \xi) \\
				\hline
				* & (t-\sqrt{2}) & \frac{\sqrt{2}}{2} t & \bar\xi\ \alpha(t) & 0 \\
				& -\sqrt{2} t & (t-\sqrt{2}) & 0 & \xi\ \alpha(t) \\
				\hline
				* & (t-\xi) & \frac{\sqrt{2}}{2}\cI(t-\xi) & \frac{\sqrt{2}}{2}(t-\sqrt{2}\cI)(t-\bar\xi) & -\frac{1}{2}(t-\sqrt{2})(t-\xi) \\
				& (t-\bar\xi) & -\frac{\sqrt{2}}{2}\cI(t-\bar\xi) & -\frac{\sqrt{2}}{2}(t-\sqrt{2})(t-\bar\xi) & -\frac{1}{2}\cI(t+\sqrt{2}\cI)(t-\xi) \\
			\end{array}\right)
	}}
	$$
We omit the first two columns since they will not be necessary for the calculation of the
twisted Alexander polynomial.
Following Wada~\cite[Lemma 1]{Wada-twisted} one needs to compute the Fitting ideal generated by the
$(4\times 4)$ minors of the submatrix of $A_1$ obtained after deleting the first two columns.
This ideal is generated by $\Delta_1(t)=\alpha(t)^2$, where $\alpha(t)=t^2-\sqrt{2}t+1=(t-\xi)(t-\bar\xi)$.
One also needs to compute $\Delta_0(t)=\det(\rho_1(x)t-I_2)=\alpha(t)$.
Hence, the twisted Alexander polynomial
$\Delta_{\rho_1,\varepsilon_1}(t)$ can be obtained as
$$
\Delta_{\rho_1,\varepsilon_1}(t)=\frac{\Delta_1(t)}{\Delta_0(t)}=\alpha(t)=(t-\xi)(t-\bar\xi).
$$
\end{proof}

\begin{rem}
Calculations for the twisted Alexander polynomials can be recreated using
SageMath\footnote{A program with the necessary source can be found in public repository:\\
\texttt{https://riemann.unizar.es/\~{}jicogo/software/AffineZariskiPair.ipynb} (Jupyter version),\\
\texttt{https://riemann.unizar.es/\~{}jicogo/software/AffineZariskiPair.sage} (plain text version).
}
The roots of $\Delta_{\rho_1,\varepsilon_1}(t)$ are in fact eigenvalues of the unitary representation
$\rho_1$ (see~\cite[Thm. 5.4]{Libgober-Non-vanishing}). It is also worth noting that the
representation $\rho_1$ factors though one of the degree-two representations of
the \emph{binary octahedral group} $\textrm{BO}_{48}$. The group $\textrm{BO}_{48}$ is the
finite group of symmetries of the octahedron and it is a non-abelian solvable
subgroup of order 48 of $\SU(2)$.
\end{rem}

\bibliographystyle{amsplain}
\providecommand{\bysame}{\leavevmode\hbox to3em{\hrulefill}\thinspace}
\providecommand{\MR}{\relax\ifhmode\unskip\space\fi MR }
\providecommand{\MRhref}[2]{%
  \href{http://www.ams.org/mathscinet-getitem?mr=#1}{#2}
}
\providecommand{\href}[2]{#2}

\end{document}